\documentclass[10pt,psamsfonts]{amsart}
\usepackage{eucal}
\usepackage{amsmath,amssymb,amsfonts,amsthm,mathrsfs}
\usepackage{graphicx, array}
\usepackage{enumerate}
\usepackage{enumitem}
\usepackage{epsfig}
\usepackage{float}
\usepackage{bbm}
\usepackage{color}
\usepackage{xcolor}
\colorlet{BLUE}{blue}
\usepackage[foot]{amsaddr}
\usepackage[T1]{fontenc}
\usepackage{enumitem}
\usepackage[colorlinks=true, urlcolor=blue, linkcolor=red]{hyperref}
\usepackage{mwe}
\usepackage{subcaption}
\usepackage[mathlines]{lineno}
\usepackage[hypcap=false]{caption}
\usepackage{listing}
\usepackage{hyperref}
\usepackage[encapsulated]{CJK}
\usepackage{caption}

\newtheorem{theorem}{Theorem}[section]
\newtheorem{lemma}{Lemma}[section]
\theoremstyle{corollary}
\newtheorem{corollary}{Corollary}[section]
\newtheorem{proposition}{Proposition}[section]
\theoremstyle{definition}

\newtheorem{remark}{Remark}[section]
\numberwithin{equation}{section}

\def\longdelete#1{}

\begin{document}

	\begin{CJK}{UTF8}{bsmi}
		\title[Propagating direction of traveling wave k=h]{Propagating Direction Near the Strong-Competition Borderline in the Two-Species Lotka–Volterra Model}
		
		\author[Chen, C-C]{Chiun-Chuan Chen$^{1,2}$}
		\address{$^1$ Department of Mathematics, National Taiwan University, Taiwan}
		\email[Chen, C-C]{$^1$chchchen@math.ntu.edu.tw}
		
		\author[Wang, S-C]{Shun-Chieh Wang$^{2,*}$}
		\address{$^2$ National Center for Theoretical Science, Taipei, Taiwan}
		\email[Wang, S-C]{$^2$rjaywang1130@ncts.ntu.edu.tw}

		
		%\footnote{These authors contributed equally to this work.}
		
		\date{\today}

		\begin{abstract}
			This study investigates the propagation direction of bistable traveling waves in the two-species Lotka–Volterra competition-diffusion model under strong competition. From an ecological perspective, the sign of the wave speed is critical, as it dictates which species eventually prevails when two species exhibit a wave-like distribution. We focus on a near-symmetric scenario where intrinsic growth rates and inter-specific competition coefficients are identical, leaving diffusion rates as the sole source of asymmetry. This framework is motivated by the conjecture “Unity is not strength” as described by Alzahrani et al. \cite{alzahrani2010travelling},\cite{alzahrani2012reversing}, Girardin and Nadin  \cite{girardin2015travelling}, and Girardin \cite{girardin2019effect}, which proposes that the species with the higher diffusion rate gains a competitive advantage, directly dictating the wave speed's sign. While extensive literature, including the significant recent progress by Nakamura and Ogiwara \cite{nakamura2025propagation}, has validated this conjecture under specific assumptions, a comprehensive proof remains elusive. In this paper, we explore the subtle regime where inter-specific competition weakens, approaching the strong-competition borderline. Leveraging our previous finding, a minimax formulation for the zero-wave-speed condition, we successfully analyze the asymptotic behavior of the wave and construct a sharp test function to determine the sign of the wave speed. Consequently, we verify  “Unity is not strength” conjecture within this new parameter regime and derive explicit bounds that characterize how the zero-wave-speed condition is influenced by the interplay between competition strength and diffusion rates.

		\end{abstract}
		
		\maketitle
		
		\section{Introduction}
		We consider the two species Lotka-Volterra
		competition-diffusion system:
		\begin{equation}\label{eq:1}
			\begin{cases}
				u_t=u_{xx}+u(1-u-cv),\\
				v_t=dv_{xx}+v(a-bu-v), \ x \in \mathbb{R}, \ t>0,
			\end{cases}
		\end{equation}
		where $d,a,b$ and $c$ are positive constants satisfying the strong competition condition $\frac{1}{c}<a<b$. Numerous studies, including \cite{gardner1982existence}, \cite{conley1984application},\cite{kan1995parameter}, and \cite{kan1996stability} have demonstrated that ($\ref{eq:1}$) admits a traveling wave solution $(u,v)$ satisfying:
		\begin{equation}\label{eq:2}
			\begin{cases}
				u''+su'+u(1-u-cv)=0,\\
				dv''+sv'+v(a-bu-v)=0, \ \xi \in \mathbb{R},\\
				(u,v)(-\infty)=(0,a), \ (u,v)(+\infty)=(1,0),
			\end{cases}
		\end{equation}
		for some wave speed $s$, where $'=\frac{d}{d\xi}$ and $\xi=x-st$. Furthermore, the wave speed $s$ is uniquely determined and satisfies
		$$-2<s<2\sqrt{ad}.$$
		Moreover, the resulting traveling wave is stable with a strictly monotone profile. 
		
		The sign of the wave speed $s$ determines the propagation direction and provides critical insight into species dominance. Specifically, $s>0$ implies the competitive exclusion of $u$ by $v$. However, determining the sign of $s$ remains a challenging task, despite significant contributions from various researchers. For instance, Rodrigo and Mimura \cite{rodrigo2000exact} constructed exact solutions under specific parameter regimes, while Guo and Lin \cite{guo2013sign} derived explicit conditions for the sign using fundamental analytic techniques. Alzahrani et al. \cite{alzahrani2010travelling} investigated the strong dispersal limit $d\to\infty$, and
		Girardin and Nadin \cite{girardin2015travelling} studied the case of very strong competition. Furthermore, Ma, Huang, and Ou \cite{ma2019speed}, Morita, Nakamura, and Ogiwara \cite{morita2023front}, and Xiao \cite{xiao2025sufficient} established criteria for $s$ via the method of upper and lower solutions. In our previous work \cite{chang2023propagating}, we introduced a new minimax formula for the zero-speed condition based on the comparison principle. For related studies, the reader is referred to Hadeler and Rothe \cite{hadeler1975travelling}. Most recently, Nakamura and Ogiwara \cite{nakamura2025propagation} made a significant advancement by developing refined upper-lower solutions to estimate the wave speed. Their work also indicated that as inter-specific competition weakens, approaching the strong-competition borderline, estimating the wave speed becomes increasingly difficult. Concurrently, Xiao \cite{xiao2025sufficient} derived sufficient conditions for the sign of wave speed in a relevant regime by considering fixed $d,c>0$ and taking the limits $b \to a^+$ and $b \to +\infty$. Building upon our previous findings, the present paper aims to address this challenge and establish new results for the borderline case of strong competition.
		
		As in \cite{chang2023propagating}, we rescale and rewrite $(\ref{eq:2})$ into the following form:
		\begin{equation} \label{maineq0}
			\begin{cases}
				u'' + su' + u(1-u)-huv = 0, \\
				dv'' + sv' + rv(1-v)-kuv = 0,
			\end{cases}
			\quad \xi \in \mathbb{R}
		\end{equation}
		under the new strong competition condition;
		\begin{equation}\label{strong}
			h>1, \ 0<r<k.
		\end{equation}
		In this case, we are interested in the traveling wave solution of (\ref{maineq0}) that satisfies
		\begin{equation}\label{maineq1}
			(u,v)(-\infty)=(0,1), \ (u,v)(+\infty)=(1,0).
		\end{equation}
		
		Kan-on's results show that $s=s(d,r,h,k)$ is a $C^1$ function of $(r,h,k)$. In addition, $s(d,r,h,k)$ is strictly increasing in $h$ and strictly decreasing in $k$. On the other hand, whether $s$ also depends monotonically on $r$ is a compelling question that remained unaddressed by Kan-on's theorem. Our previous work provided an affirmative answer and established the following results. 
		
		% temporarily use A numbering
		\renewcommand{\theproposition}{A.\arabic{proposition}}
		\renewcommand{\thelemma}{A.\arabic{lemma}}
		\renewcommand{\thetheorem}{A.\arabic{theorem}}
		\renewcommand{\thecorollary}{A.\arabic{corollary}}
		\setcounter{proposition}{0}
		\setcounter{lemma}{0}
		\setcounter{theorem}{0}
		\setcounter{corollary}{0}

		\begin{proposition}\normalfont(\cite{chang2023propagating}\label{monotoneinr})
			
			For any $d>0, \ k>0, \ h>1$, the traveling wave speed $s(d,r,h,k)$ is strictly increasing in $r \in (0,k)$. Moreover, there exists a unique $r=r_c \in (0,k)$ such that $s(d,r_c,h,k)=0$.
		\end{proposition}
		
		To determine $r_c$ for a traveling wave solution $(u,v)$, we consider $v=H(u)$ as a function of $u$ and define the following function space for $H$ based on the asymptotic behavior of the wave solution described in Kan-on's  work \cite{kan1995parameter}. Let
		\begin{equation}
			\Lambda_0=\{H \in C^2(0,1) \cap C[0,1]:H(0)=1, H(1)=0, H'(u)<0 \ \text{ for } \ 0<u<1\}.\notag
		\end{equation}
		We also propose the following conditions.
		
		\noindent{\bf (H1)} There exist $b_1\ge 0, b_2>0$ and $b_3$ (which depend on $H$) such that
		%\ \ \ $H(u)$ is $C^2$ at $u=0$ and $-\infty< H'(0)<0$ 
		
		\ $u(1-H(u))^{-1}\to b_1,\, H'(u)u(1-H(u))^{-1}\to -b_2,\, H''(u)u^{2}(1-H(u))^{-1}\to b_3$ 
		
		\ as $u\to 0^+$.
		
		\
		
		\noindent{\bf (H2)} There exist $b_4\ge 0, b_5>0$ and $b_6$ (which depend on $H$) such that
		
		\ $H(u)(1-u)^{-1}\to b_4, \,H'(u)(1-u)[H(u)]^{-1}\to-b_5,\, H''(u)(1-u)^{2}[H(u)]^{-1}\to b_6$ 
		
		\ as $u\to 1^-$.
		
		\
		
		\noindent{\bf (H3)} 
		\[
		\int_0^{u} s[1-s-hH(s)]\,ds\ \ \left \{
		\begin{array}{l}
			<0 \ \text{ for }0<u<1, \\
			=0 \ \text{ for }u=1.
		\end{array}
		\right.\ \ \ \ \ \ \ \ \ \ \ \ \ \ \ \ \ \ \ \ \ \ \ \ \ 
		\]
		
		\noindent Then we define
		
		\begin{equation}\label{functionspace_2}
			\Lambda=\{H \in \Lambda_0: H \text{ satisfies } (H1), (H2), \ \text{and } (H3)\}.
		\end{equation}
		
		In \cite{chang2023propagating}, we characterized $r_c$ via a minimax formula.
		
		\begin{theorem}\normalfont(\cite{chang2023propagating})  \label{mini-max}
			
			Let 
			\[
			R_H(u)=\frac {kuH(u)+d H'(u) u[1- u -h H(u)] + 2d H''(u)\int_0^u s[1-s -h H(s)]\,ds}{H(u)(1-H(u))}.
			\]
			Then
			\[
			r_c=\inf_{H\in \Lambda}\sup_{u\in (0,1)}R_H(u)=\sup_{H\in \Lambda}\inf_{u\in (0,1)}R_H(u).
			\]
		\end{theorem}
		
		In \cite{girardin2019effect}, Girardin considered the case 
		$h=k$, $r=1$, and $d>1$ in system $(\ref{maineq0})$. That is, he assumed $r=1$ and $d>1$ in the following system: 
		\begin{equation} \label{maineq}
			\begin{cases}
				u'' + su' + u(1-u)-kuv = 0, \\
				dv'' + sv' + rv(1-v)-kuv = 0,\\
			\end{cases}
			\quad \xi \in \mathbb{R}
		\end{equation}
		with
		\begin{equation}\label{asympto}
			(u,v)(-\infty)=(0,1), \ (u,v)(+\infty)=(1,0).
		\end{equation}
		This framework assumes that the species differ exclusively in their diffusion rates. A key point of interest is whether species $v$ can outcompete $u$ due to its higher diffusion rate. Specifically, we aim to address the following question:
		$$\text{Does } s(d,1,k,k)>0 \text{ for all }d,k>1 ? \ $$ 
		The condition $s>0$ indicates that 
		$v$ displaces $u$, thereby resulting in the competitive exclusion of species $u$
		by $v$.
		
		While a complete characterization of the sign of $s(d,1,k,k)$ for $d,k>1$ remains elusive, significant progress has been made. Studies by Rodrigo and Mimura \cite{rodrigo2000exact}, Guo and Lin \cite{guo2013sign}, Ma et al. \cite{ma2019speed}, Morita et al. \cite{morita2023front}, Chang et al. \cite{chang2023propagating}, and Nakamura and Ogiwara \cite{nakamura2025propagation} have verified that $s>0$ within specific interior regions of the parameter space. In particular, our joint work with Chang \cite{chang2023propagating} derived the following partial result by applying the minimax formula from Theorem \ref{mini-max} to determine the sign of $s$.

		\begin{corollary}\normalfont(\cite{chang2023propagating})
			Assume $k>1$; let $\gamma=\frac{2d(k+1)}{(3k-1)^2}$. Then $s(d,1,k,k)>0$ if one of the following holds
			
			\begin{align}
				&(i): \ 1<k \leq \frac{5}{3}, \ k \geq \gamma, \ k-\frac{d(k-1)}{3k-1}<1,\notag\\
				&(ii): \ 1<k \leq \frac{5}{3}, \ k < \gamma, \ \frac{4d(k-1)}{(3k-1)^2}+\frac{k(5-3k)}{2}<1,\notag\\
				& (iii): \ k > \frac{5}{3}, \ k \geq \gamma, \ \max\{\frac{4d(k-1)}{(3k-1)^2},k-\frac{d(k-1)}{3k-1}\}<1,\notag\\
				& (iv): \ k > \frac{5}{3}, \ k < \gamma, \ \frac{4d(k-1)}{(3k-1)^2}<1.\notag
			\end{align}

		\end{corollary}
		
		In contrast to the aforementioned literature, Risler \cite{risler2017competition} explored the case where $ (d,k)$ is near $(1,1)$ , while Alzahrani et al. \cite{alzahrani2010travelling} investigated the strong dispersal limit $d\to\infty$, and Girardin and Nadin \cite{girardin2015travelling} addressed the regime of the very strong competition limit $k\to\infty$. In this paper, we focus on the limit as $k\to 1^+$, a regime not covered by our previous work or the aforementioned literature. In this case, inter-specific competition weakens as it approaches the strong competition boundary, making the behavior of the traveling wave significantly more intricate and rendering a precise estimate of the wave profile challenging. Indeed, as noted by Risler \cite{risler2017competition} and Girardin \cite{girardin2019effect}, the wave speed exhibits a singularity at $k=1$ as $(d,k)$ approaches $(1,1)$.
		
		Proposition \ref{monotoneinr} ensures both the existence of a unique $r_c \in (0,k)$ such that $s(d,r_c,k,k)=0$ and the fact that $s(d,r,k,k)$ is strictly increasing in $r$. It follows from these properties that
		\[
		s(d,1,k,k)>0 = s(d,r_c,k,k)\text{ if and only if } r_c<1.
		\]
		Therefore, in the following, we focus on examining whether $r_c<1$ for $d>1$ as $k$ approaches $1$.
		
		A key observation in our study is that $u+v\to 1$ as $k\to 1^+$ in the case of vanishing wave speed ($s=0$ or $r=r_c$). Notably, while this property is straightforward to verify for $d=1$ by summing the governing equations, its persistence for $d>1$ is non-trivial. This leads to our first main result.

		%Nevertheless, the calculation becomes more complicated when $k \to 1^+$. To overcome this situation. We obtain the following new property of the traveling wave solutions in the limit $k \to 1^+$.Nevertheless, neither our previous work nor the aforementioned studies cover the regime where $k\to1^+$. In this case, inter-specific competition weakens as it approaches the strong-competition boundary, making the behavior of the traveling wave considerably more subtle and rendering a precise estimate of the wave profile formidable. While it is straightforward to prove $u+v \to 1$ as $k\to 1^+$ for $d=1$ by summing the two equations, we show that this property surprisingly holds for $d>1$ as well. Our first new result is as follows.
		
		%Normal numbering
		
		\renewcommand{\theproposition}{\thesection.\arabic{proposition}} \renewcommand{\thelemma}{\thesection.\arabic{lemma}} 
		\renewcommand{\thetheorem}{\thesection.\arabic{theorem}}
		\renewcommand{\thecorollary}{\thesection.\arabic{corollary}}
		\setcounter{proposition}{0} 
		\setcounter{lemma}{0}
		\setcounter{theorem}{0}
		\setcounter{corollary}{0}

		\begin{theorem} \normalfont\label{maintheorem0}
			For any $d>1$ and $k>1$, let $r_c=r(k) \in (0,k)$ be the value such that $s(d,r_c,k,k)=0$. Let $(u,v)(\xi)$ be the traveling wave solution of (\ref{maineq}) and (\ref{asympto}). Then
			\begin{equation}\label{sup u+v}
				\sup_{\xi \in \mathbb{R}}|u(\xi)+v(\xi)-1| \to 0^+ \ \text{ and } \ r(k) \to 1 \ \text{ as} \ k \to 1^+.
			\end{equation}
		\end{theorem}
		
		We note that the solution to (\ref{maineq})--(\ref{asympto}) is unique up to translation. Furthermore, the conclusion of (\ref{sup u+v}) remains valid under any translation of $(u,v)$. By virtue of Theorem \ref{maintheorem0}, it follows that for 
		$k$ sufficiently close to $1$, the traveling wave profile satisfies $v \approx 1-u$. This property motivates the choice of a specific ansatz, $H \approx 1-u$, in the minimax formula, which allows us to establish our main result as follows.
		
		\begin{theorem}\normalfont \label{maintheorem1}
			For $d>1$, there exists a small constant $c_T>0$ such that
			\begin{equation}
				s(d,1,k,k)>0 \  \  \text{ for }\ \  1<k<1+c_T\frac{d-1}{d^4}.
			\end{equation}
			
		\end{theorem}
		
		The theorem above provides a new result concerning the 'Unity is not strength' conjecture.
		In fact, to prove Theorem \ref{maintheorem1}, we first  obtain both an upper and a lower bound for $r_c$ as follows.
		\begin{theorem}\normalfont\label{maintheorem2}
Let $c_T$ be defined as in the previous theorem. There exist positive constants $c_L$ and $ c_U$ such that
                  \begin{equation}
			1-c_L\frac{d-1}d\theta \leq  r_c \le  1-c_U\frac{d-1}d\theta \ \text{ for } \ 0<\theta\le  c_T\frac{d-1}{d^4}.
		\end{equation}   

			%For any $d>1$, $k=1+\theta>1$, there exist positive $\tau(d)$ and $\theta_m(d)$ such that
			%\begin{equation}\label{low up r_c}
				%1-\tau(d) \theta \leq r_c \leq 1-\frac{\tau(d)}{2}\theta
			%\end{equation}
			%if $0<\theta<\theta_m(d)$.   }
		\end{theorem}

		It follows immediately from Theorem \ref{maintheorem2}, which verifies $r_c<1$, that Theorem \ref{maintheorem1} holds. A more detailed discussion about this is provided in Section 4.
		
		\begin{remark}
			In our previous work, \cite{chang2023propagating}, for fixed $d,k,h$ satisfying ($\ref{strong}$), the existence of $r_c$ was proved via the Intermediate Value Theorem. In that proof, we were unable to explicitly construct a lower bound $r_1<r_c$
			such that $s(d,r_1,h,k)<0$; instead, the existence of $r_1$ was established by a contradiction argument. This suggests that determining a lower bound for $r_c$ is highly nontrivial in general. However, the special regime $k\to 1^+$ makes it possible to derive explicit upper and lower bounds for $r_c$. 
			
		\end{remark}

		This paper is organized as follows: Section 1 introduces the main results, while Section 2 is concerned with the proof of Theorem \ref{maintheorem0}.  In Section 3, we construct a suitable ansatz $H(u) \approx 1-u$ and verify whether it belongs to the function space $\Lambda$. Finally, we utilize $H(u)$ as a test function in the minimax formula to estimate $r_c$ and determine the sign of $s(d, 1, k, k)$ as $k \to 1^+$, concluding the proofs of Theorem \ref{maintheorem1} and Theorem \ref{maintheorem2}.

		\section{Proof of Theorem \ref{maintheorem0}}
		
		Let $k=1+\varepsilon>1$ and let $r(\varepsilon) \in (0,1+\varepsilon)$ be the unique value such that $s(d,r(\varepsilon),1+\varepsilon,1+\varepsilon)=0$, where $\varepsilon>0$ is a small number. Let $(u_{\varepsilon},v_{\varepsilon})(\xi)$ denote a traveling wave solution of (\ref{maineq}) and (\ref{asympto}) with $r=r(\varepsilon)$. Since \eqref{maineq} and \eqref{asympto} are translation invariant,  $(u_{\varepsilon},v_{\varepsilon})(\xi+w)$ remains a solution for any $w \in \mathbb{R}$. By the maximum principle, $0< u_\varepsilon(\xi)<1$ and $0<v_\varepsilon(\xi) < 1$ for all $\xi \in \mathbb{R}$.
		
		First, we show that
		\begin{align}\label{Thm1.1_0}
			\sup_{\xi \in \mathbb{R}} [u_\varepsilon (\xi)+v_\varepsilon(\xi)] \leq 1+\varepsilon.
		\end{align}
		Suppose for contradiction that \eqref{Thm1.1_0} is false. By the asymptotic behavior of $(u_{\varepsilon},v_{\varepsilon})$, the function $Q(\xi):=u_{\varepsilon}(\xi)+v_\varepsilon(\xi)$ attains its maximum at some point $\xi_M \in \mathbb{R}$, implying that $\max_{\xi \in \mathbb{R}}Q(\xi)=u_{\varepsilon}(\xi_M)+v_{\varepsilon}(\xi_M)>1+\varepsilon$. Up to a translation, we may assume without loss of generality that $u_\varepsilon(\xi_M)=\frac{1}{2}$. At the maximum point, we have 
        $$Q''(\xi_M) \leq 0.$$
        On the other hand, it follows from \eqref{maineq} that $(u_{\varepsilon},v_{\varepsilon})$ satisfies

		\begin{equation}
			0=dQ''(\xi_M)+du_{\varepsilon}(\xi_M)(1-u_{\varepsilon}(\xi_M))+rv_{\varepsilon}(\xi_M)(1-v_{\varepsilon}(\xi_M))-(1+\varepsilon)(d+1)u_{\varepsilon}(\xi_M)v_{\varepsilon}(\xi_M).\notag
		\end{equation}
		Since $0<r<1+\varepsilon$, a direct evaluation at $\xi_M$ yields:
		\begin{align}
			0 &\leq du_{\varepsilon}(\xi_M)(1-u_{\varepsilon}(\xi_M))+rv_{\varepsilon}(\xi_M)(1-v_{\varepsilon}(\xi_M))-(1+\varepsilon)(d+1)u_{\varepsilon}(\xi_M)v_{\varepsilon}(\xi_M)\notag\\
			&<du_{\varepsilon}(\xi_M)(v_\varepsilon(\xi_M)-\varepsilon)+rv_{\varepsilon}(\xi_M)(u_\varepsilon(\xi_M)-\varepsilon)-(1+\varepsilon)(d+1)u_{\varepsilon}(\xi_M)v_{\varepsilon}(\xi_M)\notag\\
			&< -d \varepsilon  u_{\varepsilon}(\xi_M)v_{\varepsilon}(\xi_M)-d\varepsilon u_\varepsilon(\xi_M)-r\varepsilon v_\varepsilon(\xi_M)<0,\notag
		\end{align}
		which is a contradiction. Hence, \eqref{Thm1.1_0} must hold.

		Next, we prove the other side. Assume to the contrary that
		$$\liminf_{\varepsilon \to 0^+}\inf_{\xi \in \mathbb{R}}[u_\varepsilon(\xi)+v_\varepsilon(\xi)]=1-\delta<1$$
		for some positive constant $\delta\le 1$. By the asymptotic behavior of $(u_{\varepsilon},v_{\varepsilon})$, there exists a sequence $\varepsilon_j, j=1,2,3,...$ such that $\lim_{j \to +\infty}\varepsilon_j=0$ and the minimum of $u_{\varepsilon_j}(\xi)+v_{\varepsilon_j}(\xi)$ is achieved at some point $\xi_{\varepsilon_j}$ with $u_{\varepsilon_j}(\xi_{\varepsilon_j})+v_{\varepsilon_j}(\xi_{\varepsilon_j})<1-\frac{\delta}{2}$. Now, we consider the rescaled functions
		\[
		\begin{cases}
			U_{\varepsilon_j}(y):=u_{\varepsilon_j}(\xi_{\varepsilon_j}+y),\notag\\
			V_{\varepsilon_j}(y):=v_{\varepsilon_j}(\xi_{\varepsilon_j}+y).
		\end{cases}
		\]
		
		Since $0<r(\varepsilon_j)<k=1+\varepsilon_j$, the sequence $\{r(\varepsilon_j)\}$ is bounded. Up to a subsequence, we may assume that $r(\varepsilon_j)$ converges to some $r^*\in [0,1]$ as $j\to\infty$. Moreover, by $ (\ref{Thm1.1_0})$ and the boundedness of the $C^{2,\alpha}$ norm of $(U_{\varepsilon_j},V_{\varepsilon_j})$, there exists a subsequence of $(U_{\varepsilon_j},V_{\varepsilon_j})$, still denoted by itself, and limiting functions $(U(y),V(y))$ such that $(U_{\varepsilon_j},V_{\varepsilon_j})$ converges to $(U(y),V(y))$ in $C^2_{loc}(\mathbb{R})$ as $j\to\infty$, where $(U,V)$ satisfies the limit system:
		\begin{equation}\label{Thm1.1_-1}
			\begin{cases}
				U''+U(1-U-V)=0,\\
				dV''+r^*V(1-V)-UV=0,\\
				U(0)+V(0)=1-\delta, \\
				U(y)+V(y)\le 1,\\
				0 \leq U,V\leq 1, \\ 
				U'(y) \geq 0, \ V'(y) \leq 0.
			\end{cases}
		\end{equation}
		
		Furthermore, \eqref{Thm1.1_0} implies that 
		\begin{equation}\label{Thm1.1_1}
			1-\delta \leq U(y)+V(y) \leq 1 \ \text{ for all } y \in \mathbb{R}.
		\end{equation}
		
		\medskip
\noindent\textbf{ Step 1. We claim that $U$ and $V$ are constants.}\\
Indeed, $(\ref{Thm1.1_1})$ together with the $U$-equation in \eqref{Thm1.1_-1} implies $U''\le 0$. Since $U$ is bounded from below, its concavity forces $U' \equiv 0$, meaning $U \equiv C$ for some constant $C$. If $U(0)>0$, the $U$-equation reduces to $U(1-U-V)=0$, which implies $1-U-V=0$ everywhere, and thus $V$ is also a constant. On the other hand, if $U(0)=0$, then $U \equiv 0$. The $V$-equation then reduces to 
			\[
			V''+r^*V(1-V)=0,
			\]
			implying $V'' \le 0$. Since $V$ is bounded from below, we similarly conclude that $V' \equiv 0$ and $V$ is a constant.

			\medskip
\noindent\textbf{Step 2. A Contradiction via the test function.} \\ 
Since $U$ and $V$ are constants and the convergence $(U_{\varepsilon_j}, V_{\varepsilon_j}) \to (U, V)$ is uniform on compact sets, it follows that $U_{\varepsilon_j}(y)+(1+{\varepsilon_j})V_{\varepsilon_j}(y)$ converges uniformly to $U(y)+V(y)=1-\delta$ on any compact interval $[0,l]$ as $j \to \infty$. We choose $l > \sqrt{\frac{2}{\delta}}\pi$. Recall that $(U_{\varepsilon_j}, V_{\varepsilon_j})$ satisfies
			\begin{equation}\label{Thm1.1_4_1} 
				\begin{cases}
					U_{\varepsilon_j}''+U_{\varepsilon_j}(1-U_{\varepsilon_j}-(1+{\varepsilon_j})V_{\varepsilon_j})=0,\\
					dV_{\varepsilon_j}''+r({\varepsilon_j}) V_{\varepsilon_j}(1-V_{\varepsilon_j})-(1+{\varepsilon_j})U_{\varepsilon_j} V_{\varepsilon_j}=0,\\
					0 < U_{\varepsilon_j},V_{\varepsilon_j} < 1, \\ 
					U_{\varepsilon_j}'(y) > 0, \ V_{\varepsilon_j}'(y) < 0.
				\end{cases}
			\end{equation}
			
			For a sufficiently large $j$, the uniform convergence implies

			\begin{equation}\label{Thm1.1_5}
				|U_{\varepsilon_j}(y)+(1+{\varepsilon_j})V_{\varepsilon_j}(y)-(1-\delta)|<\frac{\delta}{2} \quad \text{for all } y \in [0,l].
			\end{equation}
			Consider the auxiliary eigenfunction $\phi(y)=\sin\left(\sqrt{\frac{\delta}{2}}y\right)$, which satisfies

			\begin{equation}\label{Thm1.1_phi_1}
				\begin{cases}
					\phi''(y)+\frac{\delta}{2}\phi(y)=0,\\
					\phi(0)=\phi(\sqrt{\frac{2}{\delta}}\pi)=0,\\
					\phi(y)>0 \text{ on } \ (0,\sqrt{\frac{2}{\delta}}\pi).
				\end{cases}
			\end{equation}
			Multiplying the $U_{\varepsilon_j}$-equation by $\phi(y)$, multiplying \eqref{Thm1.1_phi_1} by $U_{\varepsilon_j}$, and integrating the difference over $\left[0, \sqrt{\frac{2}{\delta}}\pi\right]$, we obtain

			\begin{align}
				\int_0^{\sqrt{\frac{2}{\delta}}\pi} [\phi(t)U_{\varepsilon_j}''(t)-U_{\varepsilon_j}(t)\phi''(t)]dt+\int_0^{\sqrt{\frac{2}{\delta}}\pi} [\phi(t)U_{\varepsilon_j}(t)(1-\frac{\delta}{2}-U_{\varepsilon_j}(t)-(1+{\varepsilon_j})V_{\varepsilon_j}(t))]dt=0.\notag
			\end{align}
			However, integration by parts yields
			\begin{align}
				\int_0^{\sqrt{\frac{2}{\delta}}\pi}[ \phi(t)U_{\varepsilon_j}''(t)-U_{\varepsilon_j}(t)\phi''(t)]dt&=\left[U'_{\varepsilon_j}(t)\phi(t)-U_{\varepsilon_j}(t)\phi'(t)\right]\Big|^{\sqrt{\frac{2}{\delta}}\pi}_{0}\notag\\
				&=-U_{\varepsilon_j}(\sqrt{\frac{2}{\delta}}\pi)\phi'(\sqrt{\frac{2}{\delta}}\pi)+U_{\varepsilon_j}(0)\phi'(0)>0\notag
			\end{align}
			and by \eqref{Thm1.1_5}, we also have
			\begin{align}
				\int_0^{\sqrt{\frac{2}{\delta}}\pi} \phi(t)U_{\varepsilon_j}(t)(1-\frac{\delta}{2}-U_{\varepsilon_j}(t)-(1+{\varepsilon_j})V_{\varepsilon_j}(t))dt>0.
			\end{align}
			The combination of these two positive terms contradicts the identity being zero.

			Consequently, we have shown that 
			\[
			\sup_{\xi \in \mathbb{R}} |[u_\varepsilon (\xi)+v_\varepsilon(\xi)] -1|\to 0 \text{ as }\varepsilon\to 0^+.
			\]
			It remains to prove that $\lim_{\varepsilon\to 0^+} r(\varepsilon) = 1$.

Suppose for contradiction that this limit does not hold. Then there exists a sequence $\varepsilon_j \to 0^+$ as $j \to \infty$ along which $r^*=\lim_{j\to\infty}r(\varepsilon_j) < 1$, and $v_{\varepsilon_j}(0)=\frac12$. By the same compactness arguments, $(u_{\varepsilon_j},v_{\varepsilon_j})$ converges in $C^2_{loc}(\mathbb{R})$ to a limiting profile $(U,V)$ satisfying
			\begin{equation}\label{Thm1.1_-2}
				\begin{cases}
					U''+U(1-U-V)=0,\\
					dV''+r^*V(1-V)-UV=0,\\
					V(0)=\frac12,\\
					U(y)+V(y)=1,\\
					0 \leq U,V\leq 1, \\ 
					U'(y) \geq 0, \ V'(y) \leq 0.
				\end{cases}
			\end{equation}
			Using $U''=0$, we can argue as above to show that $U'(y)=0$ and $U(y)=U(0)=1-V(0)=\frac12$ is a constant. Therefore, $V(y)=1-U(y)=\frac12$ is also a constant. By the $V$-equation of (\ref{Thm1.1_-2}),
			\[
			0=dV''+r^*V(1-V)-UV=r^*V(0)(1-V(0))-U(0)V(0)=(r^*-1)\frac14,
			\]
			which forces $r^*=1$, contradicting the assumption. This completes the proof.

		\section{Preliminary Estimates for H(u)} \label{Section 3}
		
		Let $d>1$ and let $k=1+\theta$ be fixed in (\ref{maineq}) with $\theta\approx 0^+.$ Let $r$ be the parameter such that the speed $s=s(d,r,k,k)=0$. To determine the dependence of $r$ on $d$ and $k$, we use the form $H(u)=1-u-\theta\phi(u)$ in Theorem \ref{mini-max}. In this section, we choose a specific $\phi(u)$ and establishes the fundamental properties of $H(u)$.

		\begin{lemma}\label{H1H2}
			Let $\phi(u)=\frac{au(1-u)}{bu+1}$ for some $a,b>0$. Then $H(u)=1-u-\theta \phi(u)$ satisfies conditions \textbf{(H1)} and \textbf{(H2)} for $0<\theta<\theta_{H}(a,b)$, where
			
			\begin{equation}\label{theta_o}
				\theta_{H}(a,b):=\frac{b+1}{a}.
			\end{equation}
		\end{lemma}
		
		\begin{proof}
			First, we verify condition \textbf{(H1)}. Direct calculation yields
			$$\phi'(u)=\frac{a(-bu^2-2u+1)}{(bu+1)^2} \text{ and } \phi''(u)=\frac{-2a(b+1)}{(bu+1)^3} \ \text{ for }0<u<1.$$
			Consequently, $H'(u)=-1-\theta\phi'(u)<0$ holds on $(0,1)$ provided that 
 \[
	\frac{1}{\theta} > a \left( \sup_{u \in [0,1]} \frac{bu^2+2u-1}{(bu+1)^2} \right) = \frac{a}{b+1}.
	\]

			\noindent Furthermore, direct limit evaluations at the boundaries yield:
			
			\begin{enumerate}
			\item $\displaystyle\lim_{u \to 0^+}\frac{u}{1-H(u)}=\displaystyle\lim_{u \to 0^+}\frac{1}{1+\frac{\theta a (1-u)}{bu+1}}=\frac{1}{1+\theta a}>0$.
			
			\item $\displaystyle\lim_{u \to 0^+}\frac{u H'(u)}{1-H(u)}=\displaystyle\lim_{u \to 0^+}\frac{u (-1-\theta \phi'(u))}{1-H(u)}=\frac{(-1-\theta a)}{1+\theta a}=-1<0.$
			
			\item $\displaystyle\lim_{u \to 0^+}\frac{u^2H''(u)}{1-H(u)}=\displaystyle\lim_{u \to 0^+}\frac{u^2(-\theta \phi''(u))}{1-H(u)}=0 \in \mathbb{R}$.
			\end{enumerate}
            
			Next, we verify condition \textbf{(H2)}.

			\begin{enumerate}\setcounter{enumi}{3}
			\item $\displaystyle\lim_{u \to 1^-}\frac{H(u)}{1-u}=\displaystyle\lim_{u \to 1^-}\frac{1-u-\frac{\theta au(1-u)}{bu+1}}{1-u}=1-\frac{\theta a}{b+1}>0$ if $0<\theta<\frac{b+1}{a}$.
			
			\item $\displaystyle\lim_{u \to 1^-}\frac{(1-u)H'(u)}{H(u)}=\displaystyle\lim_{u \to 1^-}\frac{(1-u)(-1-\theta \phi'(u))}{H(u)}=\frac{-b-1+\theta a}{b+1-\theta a}=-1<0.$
			
			\item $\displaystyle\lim_{u \to 1^-} \frac{(1-u)^2H''(u)}{H(u)}=\displaystyle\lim_{u \to 1^-} \frac{(1-u)}{H(u)} \cdot (1-u)(-\theta \phi''(u))=0 \in \mathbb{R}$.
			\end{enumerate}
            
			Therefore, $H(u)$ satisfies both \textbf{(H1)} and \textbf{(H2)} for any $\theta \in (0, \theta_{H}(a,b))$.
\end{proof}

		To establish condition \textbf{(H3)}, we rely on the following preparatory results.

		\begin{lemma}\label{convex}
			Let $h>1$. Assume that $H''(u)>0$ for $0<u<1$ and $F(1)=0$. Then $F(u)=\displaystyle\int_0^u s(1-s-hH(s))ds<0$ for all $0<u<1$.
		\end{lemma}
		\begin{proof}
			Define $w(s)=1-s-hH(s)$ and $n(s)=s(1-s-hH(s))$, so that $F'(u)=n(u)$ with $F(0)=F(1)=0$. 
			
			Since $H''(s)>0$ on $(0,1)$, we have $w''(s) = -hH''(s) < 0$, which implies that $w(s)$ is strictly concave on $[0,1]$. Combined with the boundary values $w(0)=1-h<0$ and $w(1)=0$, there exists a unique root $s_0 \in (0,1)$ such that $w(s)<0$ on $(0,s_0)$ and $w(s)>0$ on $(s_0,1)$. Consequently, the derivative $F'(u)$ satisfies
			\begin{equation}
				F'(u)=\begin{cases}
					n(u)<0 \ &\text{ for } \ u \in (0,s_0),\\
					n(u)>0 \ &\text{ for } \ u \in (s_0,1).
				\end{cases}\notag
			\end{equation}
			This implies that $F(u)$ strictly decreases on $(0,s_0]$ from $F(0)=0$, and strictly increases on $(s_0,1)$ toward $F(1)=0$. It immediately follows that $F(u)<0$ for all $0<u<1$.
\end{proof}

		\begin{proposition} For each $d>1$, the constant 
        $$\displaystyle\sigma_0(d):=-\int_0^1\frac{s(1-s)(2s-1)}{(d-1)s+1}ds>0.$$
		\end{proposition}
		
		\begin{proof}
			Let $f(s)=s(1-s)(2s-1)$. Then $f(s)<0$ on $(0,\frac{1}{2})$, $f(s)>0$ on $(\frac{1}{2},1)$, and $f(s)$ is anti-symmetric about $s=\frac{1}{2}$. For $d>1$, the weight function $g(s)=\frac{1}{(d-1)s+1}$ is strictly positive and strictly decreasing on $[0,1]$. Since $g(s)$ assigns greater weight to the subinterval where $f(s)<0$, the integrated product $\int_0^1 f(s)g(s)\,ds$ is strictly negative, which ensures that $\sigma_0(d)>0$.
\end{proof}

			\begin{remark}
				We note that $\sigma_0(d) \to 0^+$ as $d \to 1^+$.
			\end{remark}

		For any given $d>1$ and $\delta>0$, we define the specific perturbation profile:
		\[
		\phi(u;\delta):=\frac{(d+1+\delta)u(1-u)}{(d-1)u+1}.
		\]
		 By Lemma \ref{H1H2}, the function $H(u)=1-u-\theta \phi(u;\delta)$ satisfies \textbf{(H1)} and \textbf{(H2)} provided that $0<\theta< \theta_{H}(d+1+\delta,d-1).$ To ensure that $H(u) \in \Lambda$, we must further verify condition \textbf{(H3)}, namely,
			\[
			\int_0^{u} s[1-s-(1+\theta)H(s)]\,ds\ \ \left \{
			\begin{array}{l}
				<0 \ \text{ for }0<u<1, \\
				=0 \ \text{ for }u=1.
			\end{array}
			\right.\ \ \ \ \ \ \ \ \ \ \ \ \ \ \ \ \ \ \ \ \ \ \ \ \ 
			\]
			For this purpose, first,  we find conditions on $\delta$ and $\theta$ such that the equality holds:
			\begin{align*}
				&\int_0^{1} s[1-s-(1+\theta)H(s)]\,ds=\int_0^{1} s[1-s-(1+\theta)(1-s-\theta\phi(s;\delta))]\,ds\\
				=&\int_0^{1} s[-\theta(1-s)+\theta(1+\theta)\phi(u;\delta))]\,ds=0.
			\end{align*}
			This is equivalent to the root condition 
			\begin{equation}\label{eta}
				\eta(\delta,\theta):=\int_0^1 s\phi(s;\delta)\,ds-\frac{1}{1+\theta}\int_0^1 s(1-s)\,ds=0.
			\end{equation}
			More precisely, expanding $\eta(\delta,\theta)$ yields 
			\begin{align*}
				0=&\,\eta(\delta,\theta)=(d+1+\delta)\int_0^1 s \cdot \frac{s(1-s)}{(d-1)s+1}ds-\frac{1}{1+\theta}\int_0^1 s(1-s)ds\\
				=&\frac1{1+\theta}\displaystyle\int_0^1\frac{s(1-s)(2s-1)}{(d-1)s+1}ds+\frac{\theta(d+1)}{1+\theta}\int_0^1\frac{s^2(1-s)}{(d-1)s+1}ds+\delta\int_0^1 \frac{s^2(1-s)}{(d-1)s+1}ds\\
				=&-\frac{\sigma_0(d)}{1+\theta}+\frac{\theta (d+1)\sigma_1(d)}{(1+\theta)}+\delta\sigma_1(d),
			\end{align*}
			where $\sigma_0(d)=-\displaystyle\int_0^1\frac{s(1-s)(2s-1)}{(d-1)s+1}ds>0$ and $\sigma_1(d):=\displaystyle\int_0^1 \frac{s^2(1-s)}{(d-1)s+1}ds>0$. Hence, $\eta(\delta,\theta)=0 $ if and only if
			\begin{equation}\label{delta_1}
				\delta=\hat{\delta}(\theta,d):=\frac1{1+\theta}[\frac{\sigma_0(d)}{\sigma_1(d)}-\theta (d+1)].
			\end{equation}
			Note that $\hat{\delta}(\theta,d)$ is decreasing in $\theta$. Let 
			\[
			\theta_0(d)=\min\{1,\frac{\sigma_0(d)}{2(d+1)\sigma_1(d)} \}
                                \]
                                and 
                                \[
                                 \delta_0(d)= \hat{\delta}(\theta_0(d),d)\ge \frac1{1+1}[\frac{\sigma_0(d)}{\sigma_1(d)}-\frac{\sigma_0(d)}{2(d+1)\sigma_1(d)}  (d+1)]\ge\frac{\sigma_0(d)}{4\sigma_1(d)}.
			\]
                                Then we have the following lemma.  
			
			\begin{lemma}\label{H3=1_1}
				For $0<\theta<\theta_0(d)$, $\hat{\delta}(\theta,d)$ is a strictly decreasing function of $\theta$ and satisfies
				\[
				\frac{\sigma_0(d)}{4\sigma_1(d)}\le \hat{\delta}(\theta,d) \le \frac{\sigma_0(d)}{\sigma_1(d)}.
				\]
			\end{lemma}

			\begin{lemma}\label{C*}
				 There exist constants $\bar c_1$, $\bar c_2$, $\bar c_3$, and $\bar c_4$ such that 
\[
\bar c_1 \frac{d-1}{d^2}\le \sigma_0(d)\le \bar c_2\frac{d-1}{d^2}, \  \frac{\bar c_3}{d} \le \sigma_1(d) \le \frac{\bar c_4}{d} \ \text{ for }\ d>1.
\]

Moreover, there exist positive constants $C_*>0$ and $c_*>0$ such that 
				\[
				c_*\frac{d-1}{d}\le \frac{\sigma_0(d)}{\sigma_1(d)}\le C_*\frac{d-1}{d}\le C^* \text{ for all }d>1.
				\]
			\end{lemma}

\begin{remark}
				By this lemma,  $\theta_0(d)=O(\dfrac{d-1}{d^2})$ as $d\to 1^+$ and $d\to \infty$.
			\end{remark}

				\begin{proof}
					We first estimate $\sigma_0(d)$. By direct computation,
					\begin{align}
						\sigma_0(d)
						&=\int_0^{\frac12}
						\frac{s(1-s)(1-2s)}
						{(d-1)s+1}\,ds
						+\int_{\frac12}^{1}
						\frac{s(1-s)(1-2s)}
						{(d-1)s+1}\,ds \notag\\
						&=\int_0^{\frac12}
						\frac{s(1-s)(1-2s)}
						{(d-1)s+1}\,ds
						+\int_0^{\frac12}
						\frac{s(1-s)(2s-1)}
						{(d-1)(1-s)+1}\,ds \notag\\
						&=(d-1)
						\int_0^{\frac12}
						\frac{s(1-s)(1-2s)^2}
						{\big[(d-1)s+1\big]\big[d-(d-1)s\big]}
						\,ds.
						\notag
					\end{align}
					For any $s\in\left[0,\frac12\right]$, the denominator satisfies the lower bound:

					\[
					[(d-1)s+1]\,[d-(d-1)s]
					\ge ds\cdot \frac d2
					=\frac{d^2s}{2}.
					\]
					This yields the upper bound for $\sigma_0(d)$:
					\begin{equation}\label{sigma0:up}
						\begin{aligned}
							\sigma_0(d)
							&\le
							(d-1)
							\int_0^{\frac12}
							\frac{s(1-s)(1-2s)^2}
							{\frac{d^2s}{2}}
							\,ds  \\
							&=
							\frac{2(d-1)}{d^2}
							\int_0^{\frac12}
							(1-s)(1-2s)^2\,ds
							:=
							\bar c_2\frac{(d-1)}{d^2}.
						\end{aligned}
					\end{equation}
					Conversely, the denominator is bounded from above by 
                    \[
					[(d-1)s+1]\,[d-(d-1)s]
					\le d^2,
					\]
                    which leads to the lower bound:
					\begin{equation}\label{sigma0:low}
						\begin{aligned}
							\sigma_0(d)
							&\ge
							(d-1)
							\int_0^{\frac12}
							\frac{s(1-s)(1-2s)^2}
							{d^2}
							\,ds\\
							&=
							\frac{d-1}{d^2}
							\int_0^{\frac12}
							s(1-s)(1-2s)^2\,ds
							:=\bar c_1
							\frac{d-1}{d^2}.
						\end{aligned}
					\end{equation}
					
					Next, we estimate $\sigma_1(d)$. Since $(d-1)s+1\ge ds,$ we obtain
					\begin{equation}\label{sigma1:up}
						\sigma_1(d)
						\le
						\int_0^1
						\frac{s^2(1-s)}
						{ds}
						\,ds
						:=
						\frac{\bar c_4}{d}.
					\end{equation}
					Moreover, using $(d-1)s+1 \le d$ on $[0,\frac{1}{2}]$ gives

					\begin{equation}\label{sigma1:low}
						\sigma_1(d)
						\ge
						\int_0^1
						\frac{s^2(1-s)}
						{d}
						\,ds
						:=
						\frac{\bar c_3}{d}.
					\end{equation}
					
Combining \eqref{sigma0:up}, \eqref{sigma0:low}, \eqref{sigma1:up}, and \eqref{sigma1:low}, we conclude that
					\[
					\bar c_1\frac{d-1}{d^2}
					\le
					\sigma_0(d)
					\le
					\bar c_2\frac{(d-1)}{d^2},
					\qquad
					\frac{\bar c_3}{d}
					\le
					\sigma_1(d)
					\le
					\frac{\bar c_4}{d}.
					\]
					Consequently,
					\[
					c_*\frac{d-1}{d}
					\le
					\frac{\sigma_0(d)}{\sigma_1(d)}
					\le
					C_*\frac{(d-1)}{d},
					\]
					where $c_*=\frac{\bar c_1}{\bar c_4}$ and $C_*=\frac{\bar c_2}{\bar c_3}$. This completes the proof.
				\end{proof}

			Let 
			\[
			\theta_1(d):=\frac{d}{d+1+C_*}.
			\]
			Then for $0<\theta<\theta_0(d)$, 
			\[
			\theta_1(d)=\frac{d}{d+1+C_*} \le \frac{d}{d+1+\hat{\delta}(\theta,d)}=\theta_{H}(d+1+\hat{\delta}(\theta,d),d-1).
			\]
			Hence, by Lemma \ref{H1H2}, we are able to verify that $H(u)=1-u-\theta \phi(u;\hat{\delta}(\theta,d))$ satisfies (H1) and (H2) for $0<\theta<\min\{\theta_1(d),\theta_0(d)\}$. More precisely and furthermore, we have the following lemma.
			
			\begin{lemma}\label{H in Lambda_1}
				Let $\phi(u;\hat{\delta})=\dfrac{(d+1+\hat{\delta})u(1-u)}{(d-1)u+1}$, where $\hat{\delta}=\hat{\delta}(\theta,d)$ is defined in (\ref{delta_1}) for $0<\theta<\theta_0(d)$. Then $H(u)=1-u-\theta \phi(u;\hat{\delta}) \in \Lambda$ for all $0<\theta<\min\{\theta_1(d),\theta_0(d)\}$.
			\end{lemma}
			\begin{proof}
				As discussed above,  $H(u)$ satisfies (H1) and (H2) for $0<\theta<\min\{\theta_1(d),\theta_0(d)\}$ by Lemma \ref{H1H2}. Therefore, we only need to verify (H3). By (\ref{delta_1}) and Lemma \ref{H3=1_1}, 
				\[
				\int_0^{1} s[1-s-(1+\theta)H(s)]\,ds=\int_0^{1} s[1-s-(1+\theta)(1-s-\theta\phi(s;\hat\delta))]\,ds=0
				\]
				for $0<\theta<\min\{\theta_1(d),\theta_0(d)\}$. Therefore, by Lemma \ref{convex}, to verify that $H(u)\in\Lambda$, it suffices to show that $H''(u)>0$ for $0<u<1$. Direct differentiation yields

\begin{align*}
		H''(u) = -\theta\phi''(u;\hat{\delta}) &= -\theta (d+1+\hat{\delta}) \left[ \frac{u(1-u)}{(d-1)u+1} \right]'' \\
		&= \frac{2\theta(d+1+\hat{\delta})d}{((d-1)u+1)^3} > 0 \quad \text{for } 0<u<1.
	\end{align*}
	This completes the proof.

			\end{proof}
			
			According to our definition of $\phi$, we immediately have the following $C^2$ estimate.
			
			\begin{lemma}\label{phi bdd}
				Let $d>1$ and $\phi(u;\hat{\delta})=\dfrac{(d+1+\hat{\delta})u(1-u)}{(d-1)u+1}$, where $\hat{\delta}=\hat{\delta}(\theta,d)$ is defined in (\ref{delta_1}) for $0<\theta<\theta_0(d)$. Then there exist constants $C_0, C_1$, and $C_2$ such that
				$$ \sup_{0<u<1}\left| \phi(u)\right|\leq C_0, \sup_{0<u<1}\left| \phi'(u)\right|\leq C_1d,\text{ and }\sup_{0<u<1}\left| \phi''(u)\right|\leq C_2d^2.$$
			\end{lemma}

		\section{Proof of Theorem \ref{maintheorem1} and Theorem \ref{maintheorem2}}
		
		Let $d>1$ and $k=1+\theta$ in (\ref{maineq}). We choose the form $H(u)=1-u-\theta\phi(u)$ in Theorem \ref{mini-max}. Then a direct expansion of $R_H(u)$ with respect to $\theta$ yields
		\begin{align}
			R_H(u)&=\frac {kuH(u)+d H'(u) u[1- u -k H(u)] + 2d H''(u)\int_0^u s[1-s -k H(s)]\,ds}{H(u)(1-H(u))}\notag\\
			&=1+\theta\ \beta_1(d,u)+\theta^2\beta_2(d,u)+\beta_3(d,u,\theta)+\theta\ \beta_4(d,u,\theta)+\theta^2 \beta_5(d,u,\theta),
		\end{align}
		where 
		
		\begin{align}
			&\beta_1(d,u)=\frac{(d+1)u(1-u)+(-d+1)u\phi(u)-\phi(u)}{u(1-u)},\notag\\
			&\beta_2(d,u)=\frac{\phi^2(u)}{u(1-u)},\notag\\
			&\beta_3(d,u,\theta)=\frac{u(1-u)}{H(u)(1-H(u))}-1-\theta(\frac{(2u-1)\phi(u)}{u(1-u)})-\theta^2(\frac{\phi^2(u)}{u(1-u)}),\notag\\
			&\beta_4(d,u,\theta)=\left[\frac{(d+1)u(1-u)-(d+1)u\phi(u)}{u(1-u)}\right]\left[\frac{u(1-u)}{H(u)(1-H(u))}-1\right],\notag\\
			&\beta_5(d,u,\theta)=(S_1+\theta S_2)\left[\frac{u(1-u)}{H(u)(1-H(u))}\right],\notag\\
			&S_1=\frac{1}{u(1-u)}\left[-(d+1)u(1-u)\phi(u)+du(1-u)\phi'(u)-du\phi(u)\phi'(u)\right]\notag\\
&\  \  \  \  \  \  \  -\frac{1}{u(1-u)}\left[  2d\phi''(u)\displaystyle\int_0^u s(-1+s+\phi(s))ds\right],\notag\\
			&\text{and}\notag\\
			&S_2=\frac{1}{u(1-u)}\left[-du\phi(u)\phi'(u)-2d\phi''(u)\displaystyle\int_0^u s\phi(s)ds\right].\notag
		\end{align}
		We first observe that the linear term $\beta_1$ satisfies
		\begin{align}\label{phi_form}
			\beta_1(d,u)=-\delta<0 \iff \phi(u)=\frac{(d+1+\delta)u(1-u)}{(d-1)u+1}>0.
		\end{align}
		
		Next, we set $\phi(u)=\frac{(d+1+\delta)u(1-u)}{(d-1)u+1}$ with $\delta=\hat{\delta}(\theta,d)$ as given in (\ref{delta_1}), and establish the following uniform estimates.
		
\begin{lemma}\label{basic1}
			For  $d>1$,
\[
\sup_{0<u<1}\frac{\phi(u)}{u}=d+1+\delta, \  \sup_{0<u<1}\frac{\phi(u)}{1-u}=\frac{d+1+\delta}d, 
\]
and
\[
 \sup_{0<u<1}\phi(u)\le \frac{d+1+\delta}d, \  \sup_{0<u<1}\frac{\phi(u)}{u(1-u)}=d+1+\delta. 
\]
		\end{lemma}

\begin{proof} We have
\[
 \sup_{0<u<1}\frac{\phi(u)}{u(1-u)}= \sup_{0<u<1}\frac{d+1+\delta}{(d-1)u+1}= d+1+\delta,
\]
\[
 \sup_{0<u<1}\frac{\phi(u)}{(1-u)}= \sup_{0<u<1}\frac{(d+1+\delta)u}{(d-1)u+1}=\left.\frac{(d+1+\delta)u}{(d-1)u+1}\right|_{u=1}=\frac{d+1+\delta}d,
\]
\[
\sup_{0<u<1}\frac{\phi(u)}{u}= \left. \sup_{0<u<1}\frac{(d+1+\delta)(1-u)}{(d-1)u+1}\right|_{u=0}=d+1+\delta, 
\]
and 
\[
\sup_{0<u<1}\phi(u)\le \sup_{0<u<1}\frac{\phi(u)}{(1-u)}=\frac{d+1+\delta}d.
\]
\end{proof}

\begin{lemma}\label{basic2}
			Let $\theta_2(d)=\dfrac{d}{2(d+1+C_*)}$ with $C_*$ defined in Lemma \ref{C*}. Then for  $d>1$ and $0<\theta<\min\{\theta_0(d),\theta_1(d),\theta_2(d)\}$,
\[
\frac1{1+\dfrac{C_*(d+1+C_*)}{2(d+1)}}\leq \frac{u(1-u)}{H(u)(1-H(u))} \leq 2.
\]
		\end{lemma}

\begin{proof}
By direct computation, 
\[
A:=\frac{u(1-u)}{H(u)(1-H(u))}= \frac{u(1-u)}{(u+\theta \phi(u))(1-u-\theta \phi(u))}= \frac{1}{(1+\theta\dfrac{ \phi(u)}u)(1-\theta \dfrac{\phi(u)}{1-u})}.
\]
Hence, by Lemma \ref{basic1},
\[
A\le  \frac{1}{1-\theta_2(d) \dfrac{\phi(u)}{1-u}}\le \frac1{1-\dfrac12}=2.
\]
Lemma \ref{basic1} together with Lemma \ref{C*} and the definition of $\theta_0(d)$ implies
\[
A\ge  \frac{1}{1+\theta\dfrac{ \phi(u)}u}\ge  \frac{1}{1+\theta_0(d)(d+1+\delta)}\ge \frac1{1+\dfrac{C_*(d+1+C_*)}{2(d+1)}}.
\]
\end{proof}
In the following, we let
\[
\theta_*(d)=\min\{\theta_0(d),\theta_1(d),\theta_2(d)\}.
\]
Then by Lemma \ref{C*}, there exist positive constants ${\bar c}_L$ and ${\bar c}_U$ such that
\begin{align}\label{theta_*}
{\bar c}_L\frac{d-1}{d^2}\le \theta_*(d)\le {\bar c}_U\frac{d-1}{d^2}.
\end{align}

		\begin{lemma}\label{beta_estimate}
			For $d>1$ and  $0<\theta<\theta_*(d)$, there exists a positive constant $C_0$  such that
			$$||\theta^2\beta_2(d,u)+\beta_3(d,u,\theta)+\theta \beta_4(d,u,\theta)+\theta^2 \beta_5(d,u,\theta)||_{C(0,1)} \leq C_0d^3(\theta^2+\theta^3+\theta^4).$$
		\end{lemma}

		\begin{proof}
			
			 We assume $d>1$ and $0<\theta<\theta_*$. Let $ \left\lVert\cdot\right\rVert$ denote the norm $ \left\lVert\cdot\right\rVert_{C(0,1)}$. According to Lemma \ref{basic1} and Lemma \ref{basic2}, it is evident that
			$$||\beta_2(d,u)|| \leq c_0d$$
			for some constant $c_0>0$ and
there exists a positive constant $c_1$ such that
			$$c_1\leq \frac{u(1-u)}{H(u)(1-H(u))} \leq 2$$
			for all $0<u<1$.  For $\beta_3(d,u,\theta)$ and $\beta_4(d,u,\theta)$, we need to estimate
			\begin{align}
				&\left\lVert\frac{u(1-u)}{H(u)(1-H(u))}-1\right\rVert=\left\lVert\frac{(2u-1)\theta \phi(u)+\theta^2\phi^2(u)}{H(u)(1-H(u))}\right\rVert\\
                                          =&\left\lVert\frac{u(1-u)}{H(u)(1-H(u))}\right\Vert  \left\Vert\frac{(2u-1)\theta \phi(u)+\theta^2\phi^2(u)}{u(1-u)}\right\rVert\leq c_2d(\theta +\theta^2)\notag
			\end{align}
			for some positive constant $c_2$, where we have used Lemma \ref{basic1}, Lemma \ref{basic2}, and Lemma \ref{C*}. Thus, we have
			\begin{align}
				\left\lVert \beta_3(d,u,\theta)\right\rVert&=\left\lVert \frac{(2u-1)\theta \phi(u)+\theta^2\phi^2(u)}{H(u)(1-H(u))}-\theta(\frac{(2u-1)\phi(u)}{u(1-u)})-\theta^2(\frac{\phi^2(u)}{u(1-u)})\right\rVert\notag\\
				&=\left\lVert\frac{(2u-1)\theta \phi(u)+\theta^2\phi^2(u)}{u(1-u)}\cdot \frac{u(1-u)}{H(u)(1-H(u))}-\frac{(2u-1)\theta\phi(u)+\theta^2\phi^2(u)}{u(1-u)}\right\rVert\notag\\
				&=\left\lVert\frac{(2u-1)\theta \phi(u)+\theta^2\phi^2(u)}{u(1-u)}\left(\frac{u(1-u)}{H(u)(1-H(u))}-1\right)\right\rVert\leq c_3d^2(\theta +\theta^2)^2 \notag
			\end{align}
			for some positive constant $c_3$ and

			\begin{align}
				\left\lVert \beta_4(d,u,\theta)\right\rVert&=\left\lVert \left[\frac{(d+1)u(1-u)-(d+1)u\phi(u)}{u(1-u)}\right]\left[\frac{u(1-u)}{H(u)(1-H(u))}-1\right]\right\rVert \notag\\
&\leq c_4 d(d+1)(\theta +\theta^2)\notag
			\end{align}
			for some positive constant $c_4$.
			
			Finally, we address the estimate for $\beta_5(d,u,\theta)$. In this case, by Lemma \ref{phi bdd} and the definition of $S_1$ and $S_2$, it suffices to bound the term
			$$\hat{S}_1+\theta \hat{S}_2:=\frac{\int_0^u s(-1+s+\phi(s))+\theta s \phi(s)ds}{u(1-u)}$$
			for $0<u<1$.

			Case 1. When $u\in (0,\frac{1}{2}]$ and $0<\theta<\theta_*(d)$. By Lemma \ref{basic1}, we obtain that there exists $c_5>0$ such that
			$$ \left\lVert\hat{S}_1+\theta \hat{S}_2\right\rVert_{C(0,\frac{1}{2}]}\leq \left\lVert\frac{\int_0^u c_5\,ds}{u(1-u)}\right\rVert_{C(0,\frac{1}{2}]}\leq2 c_5.$$
            
            Case 2. When $u\in [\frac{1}{2},1)$ and $0<\theta<\theta_*(d)$. Since $H(u)$ satisfies (H3). Then we have
			\begin{align}
				\left\lVert\hat{S}_1+\theta \hat{S}_2\right\rVert_{C[\frac{1}{2},1)}=\left\lVert\frac{\int_u^1 s(1-s)-(1+\theta)s\phi(s)ds}{u(1-u)}\right\rVert_{C[\frac{1}{2},1)}\leq c_{6}, \notag
			\end{align}
			where $c_6$ is a positive constant. These estimates, together with Lemma \ref{phi bdd}, Lemma \ref{basic1}, and Lemma \ref{basic2}, imply that
                               \begin{align}
				\left\lVert \beta_5(d,u,\theta)\right\rVert&\le c_7\left\lVert d^2(1+\theta)+2d|\phi''(u)||\hat{S}_1+\theta \hat{S}_2|\right\lVert\le c_8(d^3+\theta d^2).\notag
		         \end{align}
                               for some positive constants $c_7$ and $c_8$.
      Combining the above results, we conclude that there exist constants $C_0$, $C_1$, and $C_2$ such that
			\begin{align}
&||\theta^2\beta_2(d,u)+\beta_3(d,u,\theta)+\theta \beta_4(d,u,\theta)+\theta^2 \beta_5(d,u,\theta)||_{C(0,1)} \\ \notag
\leq &\,C_1d^3\theta^2+C_2d^2(\theta^3+\theta^4)\\ \notag
\ \le &\,  C_0d^3(\theta^2+\theta^3+\theta^4) \notag
\end{align}
			for all $0<\theta<\theta_*(d)$ and $d>1$.
		\end{proof}

		We are now ready to complete the proof of Theorem \ref{maintheorem1} and Theorem \ref{maintheorem2}.
		
\

		\noindent { \bf Proof of Theorem \ref{maintheorem1} and Theorem \ref{maintheorem2}}
		
		Let $d>1$ and $k=1+\theta>1$. For $0<\theta\le \theta_*$, we employ the test function $H(u)=1-u-\theta \phi(u)$, where $\phi(u)=\dfrac{(d+1+\hat{\delta})u(1-u)}{(d-1)u+1}$ and $\hat{\delta}=\hat{\delta}(\theta,d)$ is defined by (\ref{delta_1}).
 Let 
\[
\varepsilon =\min\{1, \frac {\sigma_0(d)}{8\sigma_1(d)C_0(\theta_*+\theta_*^2+\theta_*^3)d^3}\}.
\]
Now, we assume $0<\theta\le \varepsilon\theta_*$, then
\[
C_0d^3(\theta^2+\theta^3+\theta^4)\le\theta C_0(\theta_*+\theta_*^2+\theta_*^3)d^3\varepsilon \le \theta\frac {\sigma_0(d)}{8\sigma_1(d)}.
\]
By Lemma \ref{H in Lambda_1}, $H(u) \in \Lambda$.
On the other hand, by Lemma \ref{H3=1_1}, 
\[
\frac {\sigma_0(d)}{4\sigma_1(d)}\le\hat{\delta}\le  \frac {\sigma_0(d)}{\sigma_1(d)}.
\]
Therefore, by Lemma \ref{beta_estimate} and (\ref{phi_form}), we have
		\begin{align}
			R_H(u)&=1-\hat{\delta}\theta +\theta^2\beta_2(d,u)+\beta_3(d,u,\theta)+\theta \ \beta_4(d,u,\theta)+\theta^2 \beta_5(d,u,\theta)\notag\\
		&\leq 1-\hat{\delta}\theta+C_0d^3(\theta^2+\theta^3+\theta^4)\notag\\
			&\leq 1-\hat{\delta}\theta +\frac{\sigma_0(d)}{8\sigma_1(d)}\theta \notag \\
&\leq 1-\hat{\delta}\theta +\frac{\hat{\delta}}2\theta \notag \\
&= 1-\frac{\hat{\delta}}2\theta \le 1-\frac{\sigma_0(d)}{8\sigma_1(d)}\theta\notag
		\end{align}
		for all $0<u<1$.
Similarly, we also have
\begin{align}
			R_H(u)&=1-\hat{\delta}\theta +\theta^2\beta_2(d,u)+\beta_3(d,u,\theta)+\theta \ \beta_4(d,u,\theta)+\theta^2 \beta_5(d,u,\theta)\notag\\
		&\ge 1-\hat{\delta}\theta-C_0d^3(\theta^2+\theta^3+\theta^4)\notag\\
			&\ge 1-\hat{\delta}\theta -\frac{\hat{\delta}}2\theta \notag \\
&= 1-\frac{3\hat{\delta}}2\theta \ge 1-\frac{3\sigma_0(d)}{2\sigma_1(d)}\theta\notag
		\end{align}
		for all $0<u<1$.
Now, the minimax formula from Theorem \ref{mini-max} yields
		\begin{equation}
			1-\frac{3\sigma_0(d)}{2\sigma_1(d)}\theta \leq \displaystyle\inf_{u\in(0,1)}R_H(u)\leq  r_c \le \displaystyle\sup_{u\in(0,1)}R_H(u)\leq 1-\frac{\sigma_0(d)}{8\sigma_1(d)}\theta<1.\notag
		\end{equation}
		By the monotonicity of $s(d,\cdot,h,k)$ in Proposition \ref{monotoneinr}, we have
		$$s(d,1,k,k)>s(d,r_c,k,k)=0$$
		for all $d>1$ and $0<\theta\le \varepsilon\theta_*$. By Lemma \ref{C*}, (\ref{theta_*}), and the definition of $\varepsilon$, there exists a positive constant $c_T>0$ such that 
\[
c_T\frac{d-1}{d^4}\le \varepsilon\theta_*.
\]
This implies that  for a fixed $d>1$, 
\[
s(d,1,k,k)>0 \ \ \text{ for } \ \ 1<k<1+c_T\frac{d-1}{d^4}. 
\]

Using Lemma \ref{C*}, (\ref{theta_*}), and the definition of $\varepsilon$ again, we are able to further obtain, for some positive constants $c_L$ and $c_U$, 
                  \begin{equation}
			1-c_L\frac{d-1}d\theta \leq  r_c \le  1-c_U\frac{d-1}d\theta<1 \text{ for } 0<\theta\le  c_T\frac{d-1}{d^4}.
		\end{equation}
We have proved Theorem \ref{maintheorem1} and Theorem \ref{maintheorem2}.
		
		%In fact, by the similar argument, we also have
		%\begin{corollary}
			%For any $d>1$, $k=1+\theta$, then
			%\begin{equation}\label{low up r_c}
				%1-\frac{\sigma_0(d)}{2\sigma_1(d)}\theta \leq r_c \leq 1-\frac{\sigma_0(d)}{4\sigma_1(d)}\theta.
			%\end{equation}
			%if $0<\theta<\theta_m(d)$.
		%\end{corollary}

		%\begin{remark}
		%	From the proof of Lemma \ref{keylemma}, if we let $\theta_7(d)$ tend to $0$, that is $\theta_m(d) \to 0^+$, then both the upper and lower bounds of $r_c$ will converge to 1.
		%\end{remark}

		\section{Acknowledgments}
		Chiun-Chuan Chen
		is supported by the National Science and Technology Council, Taiwan (Grant Number
		114-2115-M-002 -004 -MY3) and the National Center for Theoretical Sciences, Taiwan
		(NCTS).

		\begin{center}
			\bibliographystyle{alpha}
			\bibliography{Bibjournal.bib}
		\end{center}
		
	\end{CJK}
\end{document}